\documentclass[14pt]{article}
\usepackage[english]{babel}
\usepackage{amsfonts,amssymb,amsmath,amstext,amsbsy,amsopn,amscd,amsthm,graphicx,euscript,graphicx}
\usepackage{graphics}
\usepackage[all]{xy}

\newtheorem{theorem}{Theorem}[section]

\newtheorem{proposition}[theorem]{Proposition}
\newtheorem{corollary}[theorem]{Corollary}

\theoremstyle{definition}
\newtheorem{example}[theorem]{Example}
\newtheorem{definition}[theorem]{Definition}
\newtheorem{remark}[theorem]{Remark}

\def\Z{{\mathbb Z}}
\def\0{{\mathbf 0}}
\def\1{{\mathbf 1}}

\newcommand{\id}{\operatorname{id}}

 \def\Z{{\mathbb Z}}

\begin{document}
\title{On biquandles for the groups $G_n^k$ and surface singular braid monoid}

\author{
\smallskip\\
{\small SANG YOUL LEE}
\smallskip\\
{\small\it
Department of Mathematics, Pusan National University,
}\\
{\small\it Busan 46241, Korea}\\
{\small\it sangyoul@pusan.ac.kr}
\smallskip\\
{\small VASSILY OLEGOVICH MANTUROV}
\smallskip\\
{\small\it MIPT, Moscow,}\\
{\small\it Kazan Federal University,}\\
{\small\it North-Eastern University (China),}\\
{\small\it vomanturov@yandex.ru}
\smallskip\\
and
\smallskip\\
{\small IGOR MIKHAILOVICH NIKONOV}
\smallskip\\
{\small\it
Department of Mechanics and Mathematics, Moscow State University,
}\\
{\small\it Moscow, Russia.}\\
{\small\it nikonov@mech.math.msu.ru}}

\date{}

\maketitle

\begin{abstract}
The groups $G_n^k$ were defined by V. O.~Manturov~\cite{MN} in order to describe dynamical systems in configuration systems. In the paper we consider two applications of this theory: we define a biquandle structure on the groups $G_n^k$, and construct a homomorphism from the surface singular braid monoid to the group $G_n^2$.
\end{abstract}

\section{Introduction}

In 2015, V. O.~Manturov~\cite{MN} defined a series of groups $G_{n}^{k}$ for natural numbers $n>k\geq 1$ and formulated the following principle:

\begin{center}
{\em If dynamical systems describing the motion of $n$ particles possess a nice codimension one property depending exactly on $k$ particles, then these dynamical systems admit a topological invariant valued in $G_{n}^{k}$}.
\end{center}

These turn out to be good in studying fundamental groups of configuration and moduli spaces when the initial space admits some geometry (for example, in the Euclidean space one can study codimension 1 properties ``three points are collinear'' or ``four points belong to the same circle/line''). One of our goals is to study when we have just topological spaces.

\begin{definition}\label{def:gnk}
The groups of {\em free $k$-braids} $G_{n}^{k}$ ($n>k\geq 1$) are defined as groups with the set of generators $a_{m}$ which are indexed by all $k$-element subsets of $\{1,\ldots, n\}$,  and relations
 \begin{enumerate}
\item[(1)]
$(a_{m})^{2}=1$ for any unordered sets $m \subset \{1,\ldots, n\}, Card(m)=k$;
\item[(2)] (far commutativity) $a_{m}a_{m'} = a_{m'}a_{m},$ if $Card(m\cap m') < k-1$;
\item[(3)] (tetrahedron relation) for every set $U=\{i_1,\dots,i_{k+1}\} \subset \{1,\dots, n\}$ of cardinality $(k+1)$, let us denote them by $m^{p}=U\setminus\{i_p\}, p=1,\dots,k+1$. Then
$$a_{m^{1}}\cdots  a_{m^{k+1}} = a_{m^{k+1}}\cdots  a_{m^{1}}.$$
\end{enumerate}
\end{definition}

We shall call the words in generators $a_m\in G_n^k$ {\em diagrams of free $k$-braids}. These groups have become a sophisticated tool for researches in knot theory and low dimensional topology~\cite{FKMN}.

\begin{example} \label{examp-g2n}
Let $G^2_n$ be the group defined by the presentation generated by
$\{a_{ij}~|~\{i,j\}\subset\{1,\ldots,n\}, i < j\}$ subject to the following relations:
\begin{itemize}
\item[(1)] $a^2_{ij} = 1$ for all $i \not= j,$
\item[(2)] $a_{ij}a_{kl} = a_{kl}a_{ij}$ for distinct $i, j, k, l,$
\item[(3)] $a_{ij}a_{ik}a_{jk} = a_{jk}a_{ik}a_{ij}$ for distinct $i, j, k.$
\end{itemize}

In particular, the group $G^2_3$ is isomorphic to the group with a presentation $$<a, b, c~|~a^2 = b^2 = (abc)^2 = 1>,$$ where $a = a_{12}, b = a_{13}, c = a_{23}.$ Indeed, the relation $(abc)^2 = 1$ is equivalent to the relation $abc = cab = 1$ because of $a^2 = b^2 = c^2 = 1.$ This obviously yields all the other tetrahedron relations.
\end{example}

\begin{example} \label{examp-g3n}
Let $G^3_n$ be the group given by the presentation generated by
$\{a_{ijk}~|~\{i, j, k\} \subset \{1,\ldots,n\}, |\{i, j, k\}| = 3\}$ subject to the following relations:
\begin{itemize}
\item[(1)] $a_{ijk} = 1$ for all $\{i, j, k\} \subset \{1,\ldots,n\},$
\item[(2)] $a_{ijk}a_{stu} = a_{stu}a_{ijk}$ if $\{i, j, k\}\cap \{s, t, u\} < 2.$
\item[(3)] $a_{ijk}a_{ijl}a_{ikl}a_{jkl} = a_{jkl}a_{ikl}a_{ijl}a_{ijk}$ for distinct $i, j, k, l.$
\end{itemize}

Especially, the group $G^3_4$ is isomorphic to the group with a presentation $$<a, b, c, d ~|~ a^2 = b^2 = c^2 = 1, (abcd)^2 = 1, (acdb)^2 = 1, (adbc)^2 = 1>.$$ Here $a = a_{123}, b = a_{124}, c = a_{134},$ and $d = a_{234}.$
\end{example}

We can give a geometric interpretation to the free $k$-braids in the spirit of Artin's braids: Each free $k$-braid in $G_n^k$ given by a word $w=a_{m_1}\dots a_{m_l}, m_j=\{i_{j1},i_{j2},\dots,i_{jk}\}, j=1,\dots,l$ determines an oriented graph $\Gamma_w$ with $n$ source vertices and $n$ sink vertices of degree $1$, and $l$ vertices of degree $2k$ (with $k$ incoming and $k$ outcoming) edges, see Fig.~\ref{fig:aijk}. The vertices possesses the structure of opposite edges: Each incoming edge is assigned to a unique outcoming edge (the opposite edge).

\begin{figure}
\centering\includegraphics[width=0.2\textwidth]{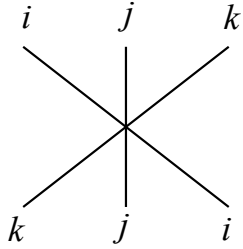}
\caption{A vertex corresponding to the generator $a_{ijk}$}\label{fig:aijk}
\end{figure}

The relations in the group $G_n^k$ induce transformations (moves) of the free $k$-braid graphs, see Figs.~\ref{fig:involution_relation}--\ref{fig:far_commutativity_relation} for the case $k=3$.

\begin{figure}
\centering\includegraphics[width=0.4\textwidth]{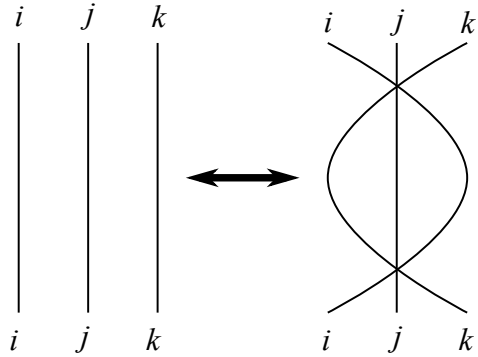}
\caption{The involution relation $a_{ijk}^2=1$}\label{fig:involution_relation}
\end{figure}

\begin{figure}
\centering\includegraphics[width=0.4\textwidth]{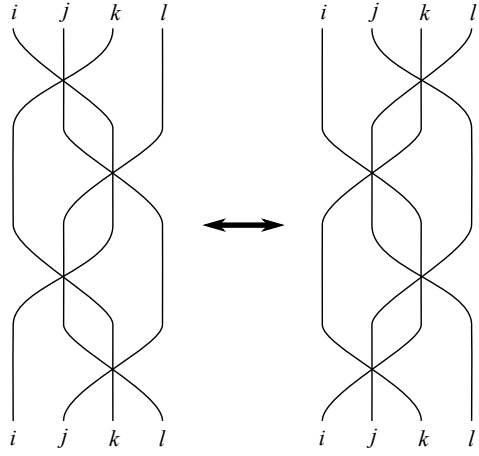}
\caption{The tetrahedron relation $a_{ijk}a_{ijl}a_{ikl}a_{jkl}=a_{jkl}a_{ikl}a_{ijl}a_{ijk}$}\label{fig:tetrahedron_relation}
\end{figure}

\begin{figure}
\centering\includegraphics[width=0.4\textwidth]{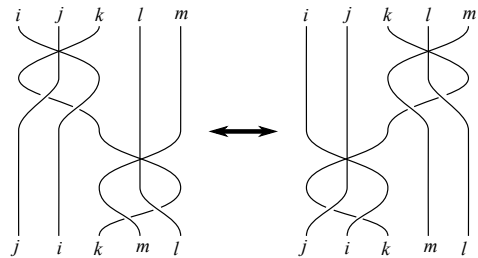}
\caption{The far commutativity relation $a_{ijk}a_{klm}=a_{klm}a_{ijk}$}\label{fig:far_commutativity_relation}
\end{figure}

\section{Biquandles for \texorpdfstring{$G_n^k$}{Gnk}}

Biquandles were introduced in~\cite{FRS}. They can be described as rules for colouring edges of link diagrams by elements of some set~\cite{EN}. These rules ensure that the sets of proper colourings of two equivalent link diagrams are isomorphic. We can take this approach to give a definition of biquandles on free $k$-braids.

\begin{definition}\label{def:biquandle_gnk}
A {\em biquandle for free $k$-braids} (or a {\em $k$-biquandle}) is a pair $(X,B)$ where $X$ is some set and $B\colon X^{\times k}\to X^{\times k}$ is a map that satisfies the relations:
\begin{enumerate}
\item $\pi\circ B=B\circ\pi$ for any permutation $\pi\in\Sigma_k$. Here we consider the natural action of $\Sigma_k$ on $X^{\times k}$ by permutation of the factors;
\item $B^2=\mathop{id}$;
\item $B_{\{1,\dots,k\}}B_{\{p+1,\dots,k+p\}}=B_{\{p+1,\dots,k+p\}}B_{\{1,\dots,k\}}$ for any $p=2,\dots,k-1$. Here $$B_{\{i_1,\dots,i_k\}}\colon X^{\times 2k-1}\to X^{\times 2k-1}$$ denotes the map which act as $B$ on the factors with indices from the set $\{i_1,\dots,i_k\}$ and as the identity on the others;
\item $B_{\hat 1}B_{\hat 2}\cdots B_{\widehat{k+1}}=B_{\widehat{k+1}}B_{\hat k}\cdots B_{\hat 1}$ where $B_{\hat i}=B_{\{1,\dots,k+1\}\setminus\{i\}}\colon X^{\times k+1}\to X^{\times k+1}$.
\end{enumerate}
\end{definition}

The definition means that $B$ induces an action $\rho$ of the group $G_n^k$ on the set $X^{\times n}$ with $\rho(a_m)=B_m$.

\begin{definition}\label{def:colouring_gnk}
Let $(X,B)$ be a biquandle on free $k$-braids and $w$ be a diagram of free $k$-braid (i.e. a word in generators $a_m$). Let $\Gamma_w$ be the geometric realization of $w$, $E=E(\Gamma_w)$ be its set of edges, and $c\colon E\to X$ be a colouring of the edges with elements of $X$. We call the colouring $c$ to be {\em good} if for each vertex of degree $2k$ the colours $(x_1,\dots, x_k)$ of the incoming edges are related to the colours $(y_1,\dots,y_k)$ of the opposite outcoming edges by the rule (see Fig.~\ref{fig:biquandle_operator})

\begin{equation}\label{eq:coloring_rule}
(y_1,\dots,y_k)=B(x_1,\dots, x_k)\equiv(b_1(x_1,\dots, x_k),\dots,b_k(x_1,\dots, x_k)).
\end{equation}

\begin{figure}
\centering\includegraphics[width=0.3\textwidth]{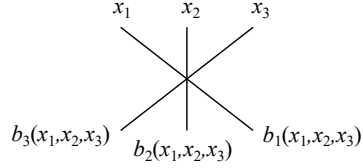}
\caption{Colouring rule of the biquandle}\label{fig:biquandle_operator}
\end{figure}
We denote the set of good colourings by $Col_{(X,B)}(w)$. Given two $n$-tuples of colours $\chi_1,\chi_2\in X^{\times n}$, we define the {\em coloring binding number} $col_{(X,B)}^{\chi_1,\chi_2}(w)$ as the number of good colourings of $w$ which have the colours $\chi_1$ on the incoming edges of the diagram $\Gamma_w$ and the colours $\chi_2$ on the outcoming edges of $\Gamma_w$.
\end{definition}
Note that $col_{(X,B)}^{\chi_1,\chi_2}(w)=0$ or $1$ because the colours of the incoming edges determine uniquely the colours of the other edges of $\Gamma_w$.

The following statement comes immediately from Definitions~\ref{def:gnk},\ref{def:biquandle_gnk} and~\ref{def:colouring_gnk}.
\begin{theorem}\label{thm:biquandle_gnk_colouring}
Let $w$ and $w'$ be diagrams of free $k$-braids. If $w$ and $w'$ define the same element in $G_n^k$ then for any biquandle for free $k$-braids $(X,B)$ there is a bijection between the sets of good colourings of $w$ and of $w'$ with elements of $X$.
\end{theorem}

Thus, the coloring binding numbers $col_{(X,B)}^{\chi_1,\chi_2}$ are invariants of free $k$-braids.

Now we give some examples of biquandles for free $k$-braids.

\begin{example}[Fundamental $k$-biquandle]
Let $w$ be a diagram of free $k$-braid and $\Gamma_w$ be its geometric realization. Analogously to the fundamental biquandle construction, we can define the {\em fundamental $k$-biquandle} $BQ_f(w)=(X_f(w),B_f(w))$ for the diagram $w$. The set $X_f(w)$ consists of equialence classes of formal words which are results of all possible applications of a formal map $B_f(w)$ to the elements of the set of edges $E(\Gamma_w)$. The equivalence of the words is induced by the relations of Definition~\ref{def:biquandle_gnk} and the equation~\ref{eq:coloring_rule}.

The fundamental $k$-biquandle can be characterized by the following statement.

\begin{theorem}
\begin{enumerate}
\item Let $w$ be a diagram of free $k$-braid. Then for any $k$-biquandle $(X,B)$ there is a natural bijection between the set of good colouring of $w$ with elements of $X$ and the set of $k$-biquandle homomorphisms from the fundamental $k$-biquandle $BQ_f(w)$ to the $k$-biquandle $(X,B)$.
\item If two diagrams $w$ and $w'$ of free $k$-braids define the same element of $G_n^k$ then their fundamental $k$-biquandles $BQ_f(w)$ and $BQ_f(w')$ are isomorphic.
\end{enumerate}
\end{theorem}
\begin{proof}
The first statement of the theorem follows from the definition of the fundamental biquandle. The second statement follows from the universal property of the fundamental biquandle (first statement of the theorem) and Theorem~\ref{thm:biquandle_gnk_colouring}.
\end{proof}

\end{example}

\begin{example}[Gaussian $k$-biquandle]
Let $X=\Z_2$ and $B(x_1,x_2,\dots,x_k)=(x_1+1,x_2+1,\dots, x_k+1)$. It is easy to see $(X,B)$ satisfies the $k$-biquandle conditions.
\end{example}

The previous example can be generalized in the following way. Let $X$ be an arbitrary set and $\tau\colon X\to X$ is an involution on $X$, i.e. $\tau^2=\id_X$. Then the map
$$B(x_1,x_2,\dots, x_k)=(\tau(x_1),\tau(x_2),\dots, \tau(x_k))$$ defines a $k$-biquandle structure on $X$ we call an {\em involution $k$-biquandle}.

We can give a further generalization.

Let $X$ be a set and $\tau\colon X\to X$ is an involution on $X$. Then $\tau$ is a product of independent transpositions $\tau=\prod_i(p_i q_i)$. For any tuple $\mathbf x =(x_1,\dots, x_k)$ we define its multiplicity vector $\mathbf m(\mathbf x)=(m_1(\mathbf x),m_2(\mathbf x),...)$ where $m_i(\mathbf x)$ is the number of elements $x_j$ in $\mathbf x$ which are equal to $p_i$ or $q_i$ for some $i$. Let
$$M_k=\{(m_1,m_2,\dots)\,|\, m_i\in\Z, m_i\ge 0,\ 1\le\sum_i m_i\le k\}$$
be the set of possible notrivial multiplicity vectors and $\mu\subset M_k$ be an arbitrary subset of it. We define a map $B_{\tau,\mu}\colon X^{\times k}\to X^{\times k}$ as following
$$
B_{\tau,\mu}(\mathbf x)=\left\{\begin{array}{cl}\tau(\mathbf x),& m(\mathbf x)\in\mu,\\
\mathbf x,& m(\mathbf x)\not\in\mu.
\end{array}\right.
$$
Here $\tau(x_1,\dots, x_k)=(\tau(x_1),\dots, \tau(x_k))$.

\begin{proposition}
The pair $(X,B_{\tau,\mu})$ is a $k$-biquandle.
\end{proposition}
\begin{proof}
The first two properties of $k$-biquandle definition are evident. For the tetrahedron and far commutativity relations note that if $\mathbf x=(x_1,\dots, x_k)$ and $\mathbf x'$ is obtained from $\mathbf x$ by applying $\tau$ to some components of $\mathbf x$ then $m(\mathbf x)=m(\mathbf x')$. This means the condition $m(\mathbf x)\in\mu$ does not change after applying operators $B_{\tau,\mu}$. Therefore, for any subsets $m$ and $m'$ with $k$ elements operators $(B_{\tau,\mu})_m$ and $(B_{\tau,\mu})_{m'}$ commute. Thus, the tetrahedron and far commutativity relations hold.
\end{proof}

We call the biquandle $(X,B_{\tau,\mu})$ a {\em conditional involution $k$-biquandle}.

The direct computation shows the following.

\begin{proposition}
Let $k=3$.
\begin{enumerate}
\item Let $X=\Z_2$. Then there is a unique up to isomorphism nontrivial $3$-biquandle on $X$. This is the Gaussian biquandle.
\item Let $X=\Z_3$. Then there are 7 up to isomorphism nontrivial $3$-biquandles on $X$. They are all conditional involution biquandles.
\end{enumerate}
\end{proposition}

\begin{example}
Let $X$ be a flat biquandle~\cite{K} with operations $(a,b)\mapsto (b\ast a, a\circ b)$, see Fig.~\ref{fig:flat_biquandle}.

\begin{figure}
\centering\includegraphics[width=0.12\textwidth]{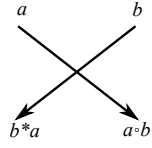}
\caption{Flat biquandle}\label{fig:flat_biquandle}
\end{figure}

Then we can define a structure of $k$-biquandle $(X,B)$ on $X$ by the splitting a $k$-crossing (see Fig.~\ref{fig:crossing_resolution}) into $\frac{k(k-1)}2$ double crossings and applying the biquandle operations to the obtained diagram. The result does not depend on the choice of the splitting.

\begin{figure}
\centering\includegraphics[width=0.4\textwidth]{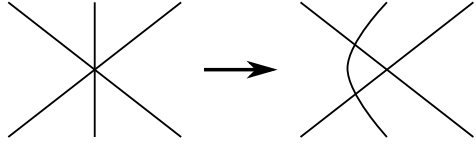}
\caption{Resolution of a $k$-crossing}\label{fig:crossing_resolution}
\end{figure}

For example, for $k=3$ we have the following formula
$$
B(x_1,x_2,x_3)=((x_1\circ x_2)\circ x_3, (x_2\ast x_1)\circ (x_3\ast(x_1\circ x_2)), (x_3\ast(x_1\circ x_2))\ast(x_2\ast x_1)).
$$

In order to the operator $B$ possesses the properties of Definition~\ref{def:biquandle_gnk} we need to impose the following relations on the flat biquandle operations:
\begin{itemize}
\item $x=(x\circ y)\ast(y\ast x)$ for all $x,y\in X$;
\item $x\ast y=x\circ y$ for all $x,y\in X$
\end{itemize}
when $k\ge 2$, and additionally
\begin{itemize}
\item $x\ast (y\circ z)=x\ast y$ for all $x,y,z\in X$;
\item $(x\circ y)\ast z = (x\circ z)\ast y$ for all $x,y,z\in X$
\end{itemize}
when $k\ge 3$.

\end{example}


\section{Marked graph diagrams and surface-links}\label{sect-mgd}

A {\it marked graph} is a graph $G$ possibly with $4$-valent vertices embedded in $\mathbb R^3$ (or $S^3$) such that it has a finite number of vertices and edges (possibly loops). Each vertex, say $v$, of $G$ is rigid, that is, there exists a small rectangular neighborhood $N(v)$ homeomorphic to $\{(x, y)|-1 \leq x, y \leq 1\},$ where $v$ corresponds to the origin and the edges incident to $v$ are represented by $x^2 = y^2$, and also has a {\it marker} represented by a thickened line segment on $N(v)$ with the core homeomorphic to $\{(x, 0)|-\frac{1}{2} \leq x \leq \frac{1}{2}\}$. A marked graph in $\mathbb R^3$ can be presented as usual by a diagram on $\mathbb R^2$, called a {\it marked graph diagram} (or {\it ch-diagram}), which is a classical link diagram on $\mathbb R^2$ possibly with marked $4$-valent vertices look like
\xy (-5,5);(5,-5) **@{-},
(5,5);(-5,-5) **@{-},
(3,-0.2);(-3,-0.2) **@{-},
(3,0);(-3,0) **@{-},
(3,0.2);(-3,0.2) **@{-},
\endxy. Throughout this paper, a classical link (diagram) is regarded as a marked graph (diagram) without marked vertices.

For a given marked graph diagram $D$, we define $L_-(D)$ and $L_+(D)$ to be the link diagrams obtained from $D$ by replacing every marked vertex \xy (-4,4);(4,-4) **@{-},
(4,4);(-4,-4) **@{-},
(3,-0.2);(-3,-0.2) **@{-},
(3,0);(-3,0) **@{-},
(3,0.2);(-3,0.2) **@{-},
\endxy with \xy (-4,4);(-4,-4) **\crv{(1,0)},
(4,4);(4,-4) **\crv{(-1,0)},
\endxy and \xy (-4,4);(4,4) **\crv{(0,-1)},
(4,-4);(-4,-4) **\crv{(0,1)},
\endxy, respectively, as illustrated in Figure~\ref{fig-nori-mg}. If $D$ has no marked vertices, then we define $L_-(D)=L_+(D)=D$. We call $L_-(D)$ and $L_+(D)$ the {\it negative resolution} and the {\it positive resolution} of $D$, respectively. A marked graph diagram $D$ is said to be {\it admissible} if the negative resolution $L_-(D)$ and the positive resolution $L_+(D)$ are all trivial link diagrams.

\begin{figure}[ht]
\begin{center}
\resizebox{0.55\textwidth}{!}{%
  \includegraphics{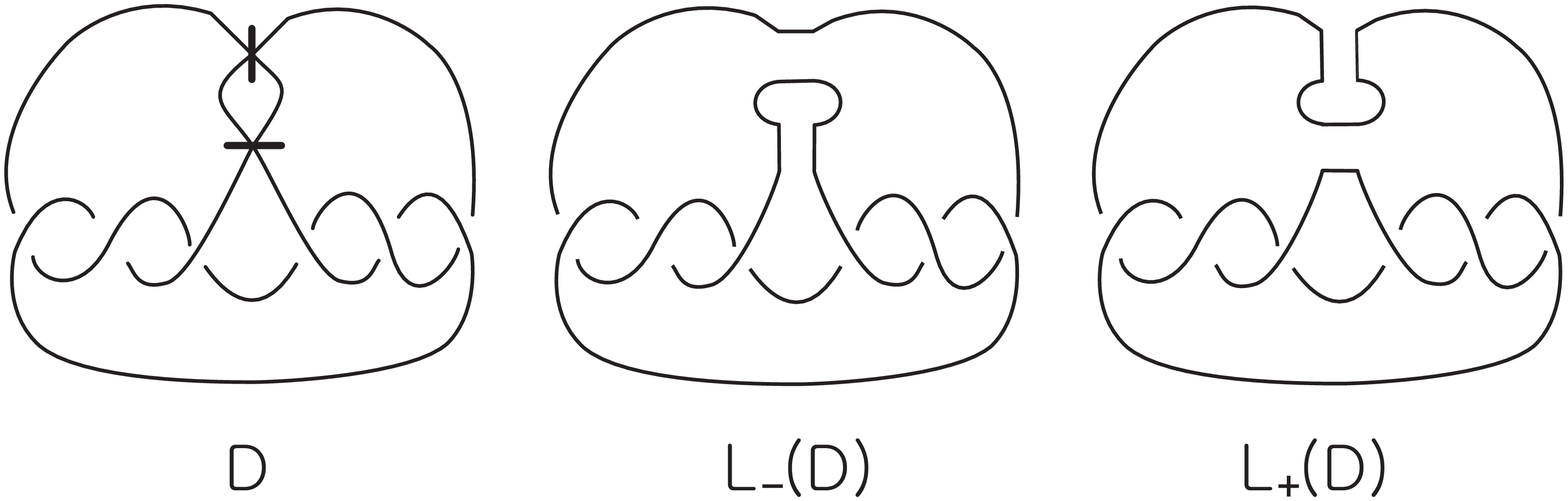}}
\caption{Resolutions of marked graph diagrams}
\label{fig-nori-mg}
\end{center}
\end{figure}

A {\it surface-link} (or {\it knotted surface}) is a closed 2-manifold smoothly (or piecewise linearly and locally flatly) embedded in the standardly oriented Euclidean $4$-dimensional space $\mathbb R^4$ (or the $4$-sphere $S^4$). A connected surface-link is called a {\it surface-knot}. A $2$-sphere-link is sometimes called a {\it $2$-link}. A connected $2$-link is also called a {\it $2$-knot}.
Two surface-links $\mathcal L$ and $\mathcal L'$ in $\mathbb R^4$ are {\it equivalent} if they are ambient isotopic, i.e., there exists an orientation preserving homeomorphism $h: \mathbb R^4 \to \mathbb R^4$ such that $h(\mathcal L)=\mathcal L'$. It is well known that every surface-link is presented by an admissible marked graph diagram and vice versa (see \cite{KSS},\cite{Lo},\cite{Yo}).

The local transformations on marked graph diagrams depicted in Figure~\ref{fig-oymoves} are called {\it Yoshikawa moves}. (These moves were first announced in \cite{Yo}.)
The Yoshikawa moves $\Omega_1, \Omega_2, \Omega_3, \Omega_4, \Omega'_4$ and $\Omega_5$ are said to be of {\it type I} and the Yoshikawa moves $\Omega_6, \Omega'_6, \Omega_7$ and $\Omega_8$ are said to be of {\it type II}. It is noted that all Yoshikawa moves of type I are realized to ambient isotopies of $\mathbb R^3$, while all Yoshikawa moves of type II are realized to ambient isotopies of $\mathbb R^4$, not $\mathbb R^3$ (see \cite{KSS}).

\begin{figure}[ht]
\begin{center}
\resizebox{0.65\textwidth}{!}{%
\includegraphics{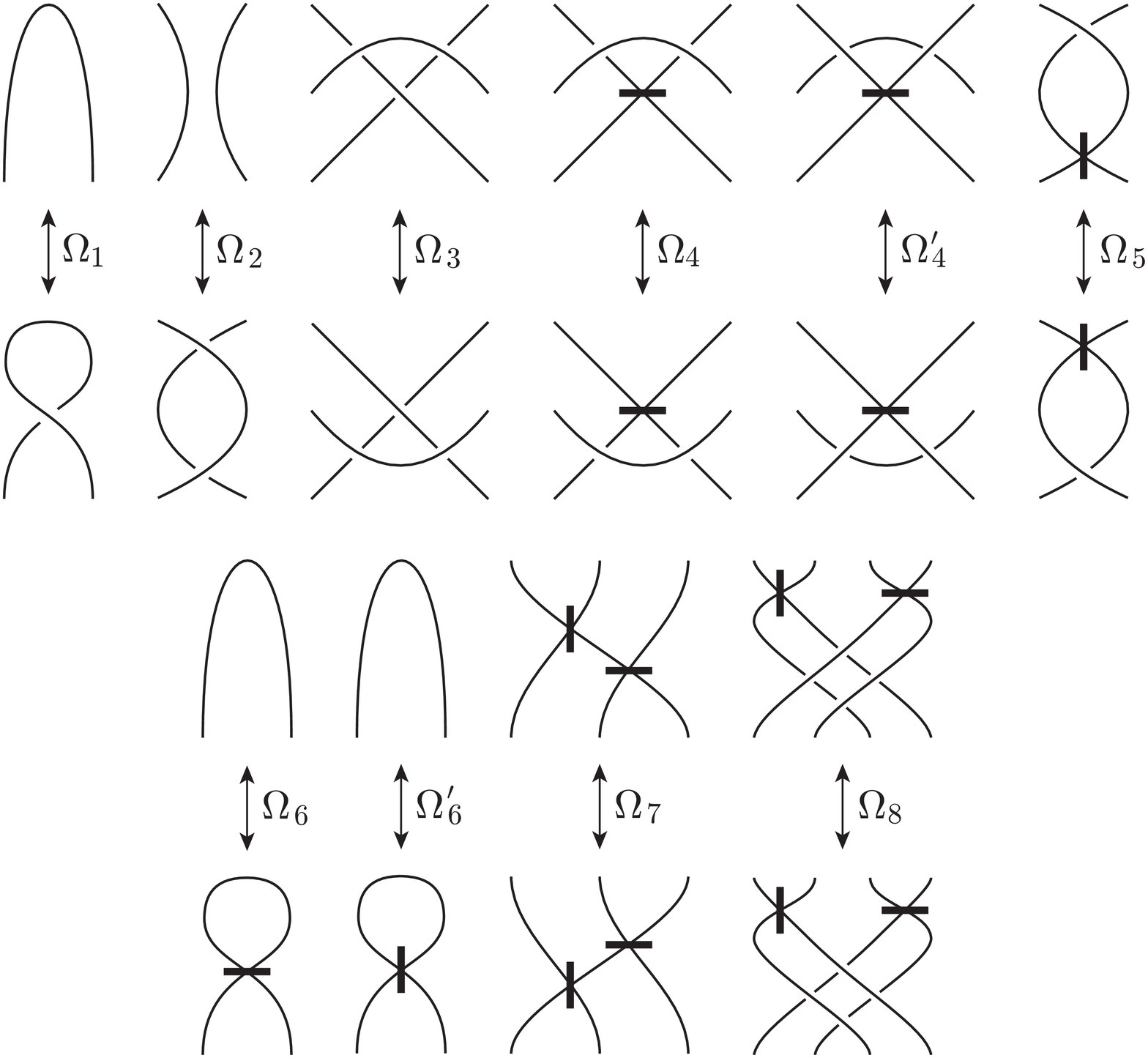} }
\caption{Yoshikawa moves}\label{fig-oymoves}
\end{center}
\end{figure}

Two marked graph diagrams $D$ and $D'$ are said to be {\it Yoshikawa move equivalent} if they can be transformed into each other by a finite sequence of Yoshikawa moves of type I, type II and ambient isotopies of $\mathbb R^2$.

\begin{theorem}[\cite{KK,KJL2,Sw,Yo}]\label{thm-equiv-mgds-ym}
Let $\mathcal L$ and $\mathcal L'$ be two surface-links and let $D$ and $D'$ be two marked graph diagrams presenting $\mathcal L$ and $\mathcal L'$, respectively. Then  $\mathcal L$ and $\mathcal L'$ are equivalent if and only if $D$ and $D'$ are Yoshikawa move equivalent.
\end{theorem}	


\section{Surface singular braid monoid \texorpdfstring{$SSB_n$}{SSBn} and groups \texorpdfstring{$G^k_n$}{Gnk}}

A marked graph diagram is said to be in {\it braid form} if it is the geometric closure of a singular braid on $n$ strands together with markers on its singular vertices for some integer $n\geq 1$. (The notion of singular braid was developed independently in \cite{Ba} and \cite{Bi}.)

\begin{proposition}[\cite{Ja1}, Proposition 3.2]
For every surface-link, there exists its marked graph diagram in braid form.
\end{proposition}

Let $SSB_n (n\geq 1)$ denote the set of all marked graph diagrams in braid form on $n$ strands. Then $SSB_n$ forms a monoid with some defining relations (see Definition \ref{defn-singularbraid}), which is called the {\it surface singular braid monoid on $n$ strands}. The product $xy$ of two surface singular braids $x$ and $y$ is obtained by putting them end to end as the product of two classical braids. For $n = 1$, this monoid is trivial with one element. Elements of $SSB_n$ are called {\it surface singular braids} (on $n$ strands), and are generated by four types of elementary surface singular braids $a_i, b_i, c_i,$ and $c_i^{-1}$ on $n$ strands for $i = 1,\ldots,n-1$, called the {\it standard generators} of $SSB_n$, where the correspondence of types of crossings and types of marked singular vertices between $i$-th and $(i+1)$-th strand are as shown in Fig.~\ref{fig-ssbgenerators}.

\begin{figure}
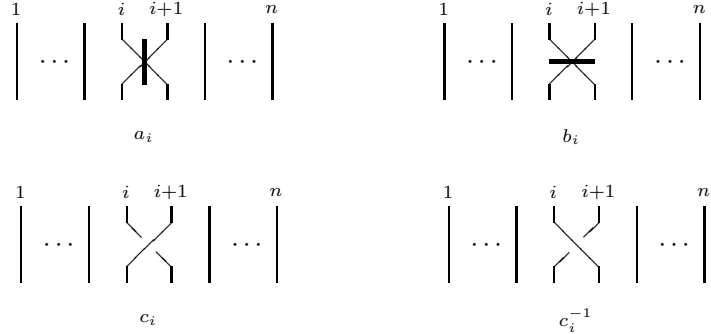

\centerline{\xy
(20,5);(20,7) **@{-},
(20,13);(20,15) **@{-},
(20,7);(26,13) **@{-},
(20,13);(26,7) **@{-},
(26,5);(26,7) **@{-},
(26,13);(26,15) **@{-},
(23,13);(23,7) **@{-},
(23.2,13);(23.2,7) **@{-},
(22.8,13);(22.8,7) **@{-},
(15,5);(15,15) **@{-}, (6,5);(6,15) **@{-},
(31,5);(31,15) **@{-}, (40,5);(40,15) **@{-},
(11,10)*{\cdots}, (36,10)*{\cdots},
(6,17)*{_{1}}, (20,17)*{_{i}}, (40,17)*{_{n}},
(26,17)*{_{i+1}}, (23,0)*{_{a_i}},
\endxy
  \qquad\qquad\qquad
\xy (20,5);(20,7) **@{-},
(20,13);(20,15) **@{-},
(20,13);(26,7) **@{-},
(26,13);(20,7) **@{-},
(26,5);(26,7) **@{-},
(26,13);(26,15) **@{-},
(20,10);(26,10) **@{-},
(20,9.8);(26,9.8) **@{-},
(20,10.2);(26,10.2) **@{-},
(15,5);(15,15) **@{-}, 
(6,5);(6,15) **@{-}, (31,5);(31,15) **@{-},
(40,5);(40,15) **@{-}, (11,10)*{\cdots}, (36,10)*{\cdots},
(6,17)*{_{1}}, (20,17)*{_{i}}, (40,17)*{_{n}},
(26,17)*{_{i+1}},(23,0)*{_{b_i}},
\endxy}
\vskip 0.5cm
\centerline{ \xy (20,5);(20,7) **@{-},
(20,13);(20,15) **@{-}, (20,7);(26,13) **@{-},
(20,13);(22,11) **@{-}, (24,9);(26,7) **@{-},
(26,5);(26,7) **@{-}, (26,13);(26,15) **@{-},
(15,5);(15,15) **@{-}, (6,5);(6,15) **@{-},
(31,5);(31,15) **@{-}, (40,5);(40,15) **@{-},
(11,10)*{\cdots}, (36,10)*{\cdots},
(6,17)*{_{1}}, (20,17)*{_{i}}, (40,17)*{_{n}},
(26,17)*{_{i+1}}, (23,0)*{_{c_i}},
\endxy
  \qquad\qquad\qquad
\xy (20,5);(20,7) **@{-}, (20,13);(20,15) **@{-},
(20,13);(26,7) **@{-}, (20,7);(22,9) **@{-},
(24,11);(26,13) **@{-}, (26,5);(26,7) **@{-},
(26,13);(26,15) **@{-},
(15,5);(15,15) **@{-}, (6,5);(6,15) **@{-},
(31,5);(31,15) **@{-}, (40,5);(40,15) **@{-}, (11,10)*{\cdots}, (36,10)*{\cdots}, (6,17)*{_{1}}, (20,17)*{_{i}}, (40,17)*{_{n}}, (26,17)*{_{i+1}}, (23,0)*{_{c_i^{-1}}},
\endxy}
{}\vspace*{5pt}\caption{The standard generators of $SSB_n$}
\label{fig-ssbgenerators}
\end{figure}

\begin{definition}(\cite{Ja1})\label{defn-singularbraid}
 Let $n \in \mathbb Z, n > 1$, and $i, j, k \in \{1,\ldots,n-1\}$ such that $|k-i| = 1$, moreover let $x_i, y_i \in \{a_i, b_i, c_i, c_i^{-1}\}.$ We define $SSB_n$ to be the monoid generated by $a_i, b_i, c_i,$ and $ c_i^{-1} (1\leq i \leq n-1)$ subjected to the following defining relations:
\begin{itemize}
\item[(A1)] $c_ic_i^{-1}=1=c_i^{-1}c_i$,
\item[(A2)] $x_iy_j = y_jx_i$ for $|i-j|>1$,
\item[(A3)] $x_ic_kc_i = c_kc_ix_k,$
\item[(A4)] $x_ic_k^{-1}c_i^{-1} = c_k^{-1}c_i^{-1}x_k,$
\item[(A5)] $a_ib_k = b_ka_i,$
\item[(A6)] $a_ib_{i-2}(c_{i-1}c_{i-2}c_ic_{i-1})^2 = a_ib_{i-2}$ for $i > 2,$
\item[(A7)] $b_ia_{i-2}(c_{i-1}c_{i-2}c_ic_{i-1})^2 = b_ia_{i-2}$ for $i > 2,$
\item[(A8)] $a^2_i= a_i,$
\item[(A9)] $b^2_i= b_i,$
\item[(A10)] $a_ib_ic^2_i= a_ib_i,$
\item[(A11)] $a_ib_k(c_ic_kc_i)^2 = a_ib_k.$
\end{itemize}
\end{definition}

For $x \in SSB_n$, we denote $\widehat{x}$ the closure of the surface singular braid $x$ obtained by connecting the $n$ initial points to the corresponding end-points by a collection of $n$ parallel strands as same as the closure of classical braids. Note that the closed surface singular braid $\widehat{x}$ is a marked graph diagram as discussed in the previous section \ref{sect-mgd}. Let $CSB_n$ be the subset of $SSB_n$ consisting of only those elements $x$ such that the closure $\widehat{x}$ is an admissible marked graph diagram, i.e., both the positive resolution $L_+(\widehat{x})$ and the negative resolution $L_-(\widehat{x})$ of $\widehat{x}$ are diagrams of trivial classical links. We consider the following Markov type relations on surface singular braids in $CSB_n$:

\begin{itemize}
\item[(C1)] $x_iS = Sx_i$ for $x_i \in \{a_i, b_i, c_i, c_i^{-1}\} (1\leq i\leq n-1)$ and $S \in CSB_n,$
\item[(C2)] $S = Sx_n$ for $S \in CSB_n$ and $x_n \in \{a_n, b_n, c_n, c_n^{-1}\} \subset CSB_{n+1}.$
\end{itemize}

\begin{proposition}(\cite{Ja1})\label{prop-markov}
Making change in a closed surface singular braid word formulation of a
surface-link by using one of the relations from (A1)-(A11) or (C1)-(C2), we receive a formula of a surface-link of the same type.
\end{proposition}

We remark that Proposition \ref{prop-markov} is a straightforward consequence of Theorem \ref{thm-equiv-mgds-ym}. The following question is still open.

\vskip 0.3cm

\noindent{\bf Question 1.} Whether any pair of marked graph diagrams in braid form of equivalent surface-link can be transformed one another by using only relations (A1)-(A11), (C1) and (C2)?

\vskip 0.3cm

The following proposition gives another useful presentation for the surface singular braid monoid $SSB_n$.

\begin{proposition} (\cite{Ja2}) \label{prop-pres-ssbm}
The monoid $SSB_n (n \geq 2)$ is generated by $a_i, b_i, c_i, c_i^{-1} (1\leq i\leq n-1)$ subjected to the following relations:
\begin{itemize}
\item[(R1)] $c_ic^{-1}_i = 1 = c_i^{-1}c_i,$
\item[(R2)] $x_iy_j = y_jx_i$ for $|i-j| > 1,$ where $x_i, y_i \in \{a_i, b_i, c_i, c_i^{-1}\}$,
\item[(R3)] $a_ic_i = c_ia_i,$
\item[(R4)] $b_ic_i = c_ib_i,$
\item[(R5)] $c_{i+1}c_ic_{i+1} = c_ic_{i+1}c_i$ for $1\leq i < n-1,$
\item[(R6)] $a_{i+1}c_ic_{i+1} = c_ic_{i+1}a_i$ for $1\leq i < n-1,$
\item[(R7)] $b_{i+1}c_ic_{i+1} = c_ic_{i+1}b_i$ for $1\leq i < n-1,$
\item[(R8)] $a_ic_{i+1}c_i = c_{i+1}c_ia_{i+1}$ for $1\leq i < n-1,$
\item[(R9)] $b_ic_{i+1}c_i = c_{i+1}c_ib_{i+1}$ for $1\leq i < n-1,$
\item[(R10)] $a_ib_{i+1} = b_{i+1}a_i$ for $1\leq i < n-1,$
\item[(R11)] $a_ib_i = b_ia_i,$
\item[(R12)] $a_i^2 = a_i,$
\item[(R13)] $b_i^2 = b_i,$
\item[(R14)] $a_ib_ic_i^2 = a_ib_i,$
\item[(R15)] $a_ib_{i+1}(c_ic_{i+1}c_i)^2 = a_ib_{i+1}$ for $1\leq i < n-1,$
\item[(R16)] $a_ib_{i+2}(c_{i+1}c_ic_{i+2}c_{i+1})^2 = a_ib_{i+2}$ for $1\leq i < n-2.$
\end{itemize}
\end{proposition}

In the rest of the paper, we shall investigate some relationship between the surface singular braid monoid $SSB_n$ and the groups $G_{m}^{k}$ of free $k$-braids, and also $k$-biquandles for $G_m^k$ for some $m$ and $k$?. It is natural to ask the following questions.

\vskip 0.3cm

\noindent{\bf Question 2.} Whether or not there is a homomorphism from $SSB_n$ to $G_{m}^{k}$ for some $m$ and $k$?

\vskip 0.3cm

We give a positive answer to this question in the next section.




\section{Virtual surface singular braid monoid \texorpdfstring{$VSSB_n$}{VSSBn}}

In this section we extend the surface singular braid monoid by adding a new type of crossing --- a virtual crossing (see Fig.~\ref{fig:virtual_crossing}). Then we define a homomorphism from the constructed virtual surface singular braid monoid to the group $G^2_n$ and describe the coloring invariants virtual surface singular braids induced by this homomorphism.

\begin{figure}
\centering\includegraphics[width=0.2\textwidth]{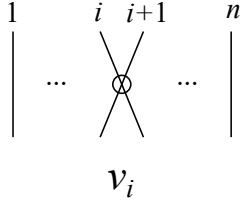}
\caption{Virtual crossing}\label{fig:virtual_crossing}
\end{figure}

\begin{definition}
Let $n\ge 2$ be an integer. The {\em virtual surface singular braid monoid} $VSSB_n$ is the monoid with the generators $a_i, b_i, c_i, c_i^{-1}$ and $v_i$, $1\leq i \leq n-1$, and the relations (A1)-(A11) from Definition~\ref{defn-singularbraid} and the virtual relations:
\begin{itemize}
\item[(V1)] $v_ix_j = x_jv_i$ for $|i-j|>1$ and $x_j\in\{a_j, b_j, c_j, c_j^{-1}, v_j\}$,
\item[(V2)] $v_i^2=1$,
\item[(V3)] $x_iv_{i\pm1}v_i = v_{i\pm1}v_ix_{i\pm1}$, $x_j\in\{a_j, b_j, c_j, c_j^{-1}, v_j\}, j=i,i\pm1$
\item[(V4)] $x_iv_i=v_ix_i,$ $x_i\in\{a_i,b_i\}$.
\end{itemize}
\end{definition}

\begin{remark}
Yoshikawa moves for virtual surface-links were found by L. Kauffman~\cite{K1}. Here we adapt those moves to the braid case. The relations (V2),(V3) correspond to the detour move of Kauffman, and the relation (V4) is the commutation relation between virtual crossings and marks.
\end{remark}

Let $\Sigma_n$ be the group of permutations of the set $\{1,\dots,n\}$. There is a homomorphism $\rho_n$ from the monoid $VSSB_n$ to the group $\Sigma_n$:
$$
\rho_n(a_i)=\rho_n(b_i)=1,\quad \rho_n(c_i)=\rho_n(c^{-1}_i)=\rho_n(v_i)=(i\ i+1),\quad i=1,\dots,n-1.
$$
Here $(i\ i+1)$ is the transposition of elements $i$ and $i+1$.

The kernel of the homomorphism $\rho_n$ is a submonoid $VPSSB_n$ in $VSSB_n$. We call it the {\em virtual pure surface singular braid monoid}.

Now, let us define an action of the monoid $VSSB_n$ on the product $G_n^2\times\Sigma_n$. Given an element $(g,\sigma)\in G_n^2\times\Sigma_n$, we define
\begin{gather*}
c_i\cdot(g,\sigma)= c_i^{-1}\cdot(g,\sigma)=(a_{\sigma(i),\sigma(i+1)}g,\sigma\cdot(i\ i+1)),\\
a_i\cdot(g,\sigma)=b_i\cdot(g,\sigma)=(g,\sigma),\quad v_i\cdot(g,\sigma)=(g,\sigma\cdot(i\ i+1)).
\end{gather*}

Then for any $\beta\in VSSB_n$ we have
\begin{equation}\label{eq:VSSB_Gn2_map}
\beta\cdot (1,1)=(\phi_n(\beta),\rho_n(\beta)).
\end{equation}

The direct computations show
\begin{proposition}
1. The map $\phi_n$ defines is a well-defines mapping of the monoid $VSSB_n$ to the set $G_n^2$.

2. The restriction of $\phi_n$ to the  virtual pure surface singular braid monoid $VPSSB_n$ is a homomorphism to the group $G_n^2$.
\end{proposition}

\begin{corollary}
Let $(X,B)$ be a $2$-biquandle and $\chi_1,\chi_2\in X^{\times n}$. Then the colouring binding number $col_{(X,B)}^{\chi_1,\chi_2}(\phi_n(\beta))$, $\beta\in VSSB_n$, is an invariant of virtual surface singular braids.
\end{corollary}

\begin{example}
Consider the virtual surface singular braid $\beta=c_1a_2v_3c_2^{-1}b_1c_2^{-1}$, see Fig.~\ref{fig:vssb}. Then
$\phi_4(\beta)=a_{12}a_{23}^2=a_{12}$. Take the Gaussian $2$-biquandle $X=\Z_2$, $B(x_1,x_2)=(x_1+1,x_2+1)$, and choose the colours $4$-tuples $\chi_1=(0,1,0,1)$ and $\chi_2=(0,0,1,1)$. Then $col_{(X,B)}^{\chi_1,\chi_2}(\phi_4(\beta))=1$.

\begin{figure}
\centering\includegraphics[width=0.18\textwidth]{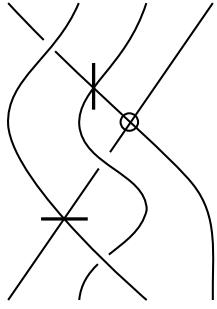}
\caption{A virtual surface singular braid}\label{fig:vssb}
\end{figure}
\end{example}

Although we constructed a homomorphism from surface singular braids to $G^2_n$, it is trivial in the sense that $\phi_n(\beta)=1$ if $\beta$ is a classical braid, i.e. has no virtual crossings. Thus, the following question remains actual.

\vskip 0.3cm

\noindent{\bf Question 2'.} How to construct a nontrivial homomorphism from $SSB_n$ to $G_{m}^{k}$ for some $m$ and $k\ge 3$?

\section*{Acknowledgements}
The first author was supported under the framework of international cooperation program managed by the National Research Foundation of Korea (NRF-2019K2A9A1A06100201). The second and the third authors were supported by the Russian Foundation for Basic Research (19-51-51004-NIF-a).


 \end{document}